\documentclass[11pt]{article}

\usepackage{hyperref}

\usepackage{array}
\usepackage{paralist}
\usepackage{graphicx}
\usepackage{amsmath}
\usepackage{amssymb}
\usepackage{amsthm}
\usepackage{mathtools}
\usepackage{setspace}

\newtheorem{thm}{Theorem}[section]
\newtheorem{claim}[thm]{Claim}
\newtheorem{lem}[thm]{Lemma}
\newtheorem{define}[thm]{Definition}
\newtheorem{cor}[thm]{Corollary}

 \newtheorem{notation}[thm]{Notation}

\newcommand{\comment}[1]{}
\newcommand{\remove}[1]{}

\newcommand{\comments}[1]{}

\newcommand{\Z}{\mathbb{Z}}

\renewcommand{\Pr}{\mathbf{Pr}}

\newcommand{\ip}[2]{\langle #1,#2 \rangle}

\DeclareMathOperator*{\E}{\mathbb{E}}
\newcommand{\MV}{\text{MV}}
\renewcommand{\vec}[1]{\boldsymbol{#1}}
\newcommand{\vmod}[2]{#1^{(#2)}}
\newcommand{\vdiv}[2]{#1^{[#2]}}

\title{\bf Matching-Vector Families and LDCs Over Large Modulo}
\author{Zeev Dvir\thanks{Department of Computer Science and Department of Mathematics, Princeton University, Princeton NJ.
Email: \texttt{zeev.dvir@gmail.com}. Research partially
supported by NSF grants CCF-0832797, CCF-1217416 and by the Sloan fellowship.}, Guangda Hu\thanks{Department of Computer Science  Princeton University, Princeton NJ.
Email: \texttt{guangdah@cs.princeton.edu}. Research partially
supported by NSF grants CCF-0832797, CCF-1217416 and by the Sloan fellowship.}}
\date{}

\begin{document}
\maketitle
\begin{spacing}{1.5}

\begin{abstract}
We prove new upper bounds on the size of families of vectors in $\Z_m^n$ with restricted modular inner products, when $m$ is a large integer. More formally, if $\vec{u}_1,\ldots,\vec{u}_t \in \Z_m^n$ and $\vec{v}_1,\ldots,\vec{v}_t \in \Z_m^n$ satisfy $\langle\vec{u}_i,\vec{v}_i\rangle\equiv0\pmod m$ and  $\langle\vec{u}_i,\vec{v}_j\rangle\not\equiv0\pmod m$ for all $i\neq j\in[t]$,  we prove that $t \leq O(m^{n/2+8.47})$. This improves a recent bound of $t \leq m^{n/2 + O(\log(m))}$ by  \cite{BDL13} and is the best possible up to the constant $8.47$ when $m$ is sufficiently larger than $n$. 

The maximal size of such families, called `Matching-Vector families', shows up in recent constructions of  locally decodable error correcting codes (LDCs) and determines the rate of the code. Using our result we are able to show that these codes, called Matching-Vector codes, must have encoding length at least $K^{19/18}$ for $K$-bit messages, regardless of their query complexity. This improves a known super linear bound of $ K2^{\Omega({\sqrt{\log K}})}$ proved in \cite{DGY11}.
\end{abstract}

\section{Introduction}

A Matching-Vector family (MV family) in $\mathbb{Z}_m^n$ is defined as a pair of ordered lists  $U=(\vec{u}_1,\ldots,\vec{u}_t)$ and $V=(\vec{v}_1,\ldots,\vec{v}_t)$  with $\vec{u}_i,\vec{v}_j\in\mathbb{Z}_m^n$, satisfying the following property: for all $i\in[t]$, $\langle\vec{u}_i,\vec{v}_i\rangle\equiv0\pmod m$ whereas for all $i\neq j\in[t]$, $\langle\vec{u}_i,\vec{v}_j\rangle\not\equiv0\pmod m$. Here $\langle\cdot,\cdot\rangle$ denotes the standard inner product. If one restricts the entries of the vectors in the family to be in the set $\{0,1\}$ the inner products corresponds to the sizes of the intersections (modulo $m$) and, in this case, MV families are more commonly referred to as families of sets with restricted modular intersections. MV families were studied previously in the context of Ramsey graphs \cite{Gro00}, circuit complexity \cite{BBR} and, more recently, were used to construct Locally Decodable Codes (LDCs) \cite{Yek08, Efr09,DGY11}, which are error correcting codes with super-efficient decoding properties. We will elaborate more on the connection to LDCs after we state our results.

We denote by $\MV(m,n)$  the size of the largest MV family in $\mathbb{Z}_m^n$ (the size of the family is $t$ in the above notation). It is an interesting (and mostly open) question to determine the value (or even order of magnitude) of $\MV(m,n)$ for arbitrary $m$ and $n$. Upper and lower bounds on $\MV(m,n)$ can be roughly divided into two kinds, corresponding to the relative size of the two parameters. One typical regime is when $m$ is small and $n$ tends to infinity and the other is when $m >> n$ (of course there are intermediate scenarios as well). 

Although our work focuses on the regime when $m$ is much larger than $n$, we first describe the known results for the other regime, namely when $m$ is a fixed constant and $n$ tends to infinity. These regime is further  divided  into the case when $m$ is prime and when $m$ is composite. When $m$ is a small {\em prime}  and $n$ tends to infinity, the value of $\MV(m,n)$ is known to be of the order of $n^{m-1}$ \cite{BF98}. When $m$ is a small {\em composite}, the picture is very different and there are exponential gaps between known lower and upper bounds on $\MV(m,n)$. A surprising construction by Grolmuzs \cite{Gro00} shows that $\MV(m,n) \geq \exp\left(c\cdot\frac{\log(n)^r}{(\log\log n)^{r-1}}\right)$ when $m$ has $r$ distinct prime factors (here $c$ is an absolute constant). That is, $\MV(m,n)$ can be super-polynomial in $n$ (that is $n^{\omega(1)}$)  for $m$ as small as $6$ (compared with the polynomial upper bound $n^{m-1}$ for prime $m$). A trivial upper bound on $\MV(m,n)$ is $m^n$ since an MV family cannot contain the same vector twice. The best upper bound on $\MV(m,n)$ for small composite $m$ was proved in \cite{BDL13} and is $m^{n/2 + O(\log m)}$. Assuming the Polynomial-Freiman-Ruzsa (PFR) conjecture \cite{TV} this can be improved to $\MV(m,n) \leq C_m^{ n/\log(n)}$ with $C_m$ a constant depending only on $m$. 

\begin{table}[htb]
\centering
\begin{tabular}{|c|c|}
\hline
$m$ & upper bound for $\MV(m,n)$ \\
\hline
general prime & $O(m^{n/2})$ \cite{DGY11} \\
\hline
small, fixed prime & $O(n^{m-1})$ \cite{BF98} \\
\hline
general composite & $m^{n/2+O(\log m)}$ \cite{BDL13} \\
\hline
small, fixed composite & $2^{O_m(n/\log n)}$ \cite{BDL13} (assuming PFR) \\
\hline
general composite & $O(m^{n/2+8.47})$ (Theorem~\ref{thm:main}) \\
\hline
\end{tabular}
\caption{List of upper bounds on $\MV(m,n)$}
\end{table}

{\bf Our work} focuses on the regime when $m$ is larger than $n$. In this setting, a construction 
of \cite{YGK12} gives $\MV$ families of size $\left(\frac{m+1}{n-2}\right)^{n/2-1}$ \cite{YGK12}. For a large {\em prime} $m$, this construction almost matches an upper bound of $O(m^{n/2})$ proved in \cite{DGY11}. For composite $m$, the best upper bound on $\MV(m,n)$ for large $m$ prior to this work was the same $m^{n/2+O(\log m)}$ bound from \cite{BDL13}. Notice that, when $m > 2^n$,  this bound is meaningless since it exceeds the trivial bound of $m^n$.  In this work we extend the proof method developed in \cite{BDL13} to give the following bound:

\begin{thm} \label{thm:main} 
	For all integers $m >1, n$ we have  $\MV(m,n) \leq 100m^{n/2+8.47}$. When $m$ is a product of distinct primes the constant $8.47$ can be replaced with $4+o(1)$.
\end{thm}

For small $n$, this bound is tight up to the constant $8.47$ as the \cite{YGK12} construction shows. When $m$ is small, this still gives some improvement over the $m^{n/2 + O(\log m)}$ bound of \cite{BDL13} but not as dramatic (and probably far from being tight).

The main tool in our proof is Fourier analysis in the spirit of \cite{BDL13}, with which we repeatedly reduce $m$ to one of its factor (eventually reaching the case of $m=1$). The distribution of $\ip{\vec{v_i}}{\vec{u_j}}$ over random $i,j \in [t]$ is far from the uniform distribution (since the probability of obtaining zero is small). This fact is used to find a large coefficient in its Fourier spectrum. This coefficient is then used to carve out a large sub family which is again an MV family, but over some proper factor of $m$. The proof ends when we reach the case of prime $m$. The difference between our proof and the one in \cite{BDL13} is in the choice of the large coefficient (or character). We are able to show that a large character appears that has nicer number theoretic properties and so are able to analyze the loss in each step in a better way -- getting rid of the $O(\log m)$ factor in the exponent.

\subsection{$\MV$ families and Locally Decodable Codes}

A $(q,\delta,\epsilon)$-Locally Decodable Code, or LDC, encodes a $K$-symbol message $x$ to an $N$-symbol codeword $C(x)$, such that every symbol $x_i$ ($i\in[K]$) can be recovered with probability at least $1-\epsilon$ by a randomized decoding procedure that makes only $q$ queries to $C(x)$, even if $\delta N$ locations of the codeword $C(x)$ have been corrupted. Understanding the minimum length $N=N(k)$ of an LDC with constant $q$ is a central research question that is still far from being solved. For $q=1,2$, this question is completely answered.  There are no LDCs for $q=1$ \cite{KT00} and the best LDCs for $q=2$ have exponential length \cite{GKST02, KdW04}. However, for $q>2$ there are huge gaps between lower bound and LDC constructions. The best known lower bound is $N=\tilde{\Omega}(K^{1+1/(\lceil r/2\rceil-1)})$ for $k>4$ \cite{Woo07} and $N=\Omega(K^2)$ for $k=3,4$ \cite{KdW04,Woo10}, while the best construction has super-polynomial length. Constructions of LDCs have been studied extensively for more than a decade. Until recently, all constructions of LDC with constant $q$ had exponential encoding length. In a breakthrough work of Yekhanin \cite{Yek08} and following improvements
\cite{Efr09,Rag07,KY09,IS10,CFL+10,DGY11,BET10}, a new family of LDCs  based on Matching Vector families was introduced. These codes, called Matching-Vector codes, rely on constructions of MV families and can have sub-exponential length for $q$ as small as 3 \cite{Efr09}. Using Grolmuzs construction as a building block, one obtains an encoding length of roughly $$N\sim\exp\exp\left((\log K)^{O(\log\log q/\log q)}(\log\log K)\right).$$ The size of the MV family used in the code construction is critical. In its simplest form, an MV code using an MV family of size $t$ in $\Z_m^n$ will send $K=t$ bits of message into $N = m^n$ bits of encoding and will require $q=m$ queries to decode. Several improvements are possible for reducing the number of queries below $m$ but these are case-based and hard to generalize for arbitrary $m$. 

Our improved bound on the size of MV families allows us to prove an unconditional lower bound on the encoding length of MV codes, {\em regardless of the query complexity}. 

\begin{thm}\label{thm-MVcodes}
For any  $MV$-code with message length $K$ and codeword length $N$ we have $N > K^{\frac{19}{18}}$. This bound is regardless of the number of queries.
\end{thm}

This theorem improves on a bound of $N > K2^{\Omega({\sqrt{\log K}})}$ proved in \cite{DGY11}.
 
\subsection{Organization}
We begin in Section~\ref{sec-prel} with a number of preliminary lemmas and notations that will be used throughout the proof. In Section~\ref{sec-prooflem} we prove our main technical lemma which is the heart of our proof. The lemma is used iteratively in the proof of our main theorem which is given in Section~\ref{sec-proofthm}. The proof of the stronger bound for the case when $m$ is a product of distinct primes is given in the appendix.

\section{Preliminaries}\label{sec-prel}

\subsection{Fourier lemma}

We consider a probability distribution $\mu$ over $\mathbb{Z}_m$. Let $\omega_m=e^{\frac{2\pi}{m}i}$ be an order $m$ primitive root of unity. It is not difficult to see $\E_{x\sim\mu}[\omega_m^{jx}]=0$ for all $j\in\{1,2,\ldots,m-1\}$ if $\mu$ is the uniform distribution. We will show that $\E_{x\sim\mu}[\omega_m^{jx}]$ is bounded away from zero for some $j\in\{1,2,\ldots,m-1\}$ if $\mu$ is far from being uniform.

In \cite{BDL13}, it was shown that
\[\max_{1\leq j\leq m-1}\left|\E_{x\sim\mu}[\omega_m^{jx}]\right|=\Omega(\frac{1}{m^{1.5}})\]
if the statistical distance between $\mu$ and the uniform distribution is big, i.e.
$\frac{1}{2}\sum_{x\in\mathbb{Z}_m}|\mu(x)-\frac{1}{m}|=\Omega(\frac{1}{m}).$
In the following lemma, we prove a better lower bound that depends only on $s=\text{order}(\omega_m^j)$ under a stronger condition 
$|\mu(0)-\frac{1}{m}|=\Omega(\frac{1}{m}).$

Consider $\mu$ as a function from $\mathbb{Z}_m$ to $\mathbb{C}$. For $0\leq j\leq m-1$, the Fourier coefficient $\hat{\mu}(j)$ is
\[\hat{\mu}(j)=\frac{1}{m}\sum_{x\in\mathbb{Z}_m}\mu(x)\omega_m^{-jx}=\frac{1}{m}\E_{x\sim\mu}[\omega_m^{-jx}].\]
One can see that $\hat{\mu}(0)=\frac{1}{m}$. The set of functions $\{\omega_m^{jx}\mid0\leq j\leq m-1\}$ is an orthogonal basis for all functions from $\mathbb{Z}_m$ to $\mathbb{C}$, and the function $\mu(x)$ can be written as
\begin{equation} \label{eq:fourier}
\mu(x)=\sum_{j=0}^{m-1}\hat{\mu}(j)\omega_m^{jx}.
\end{equation}

\begin{lem} \label{lem:bias}
Let $\mu:\mathbb{Z}_m\mapsto[0,1]$ be a probability distribution over $\mathbb{Z}_m$ (i.e. $\sum_{x\in\mathbb{Z}_m}\mu(x)=1$). If $\mu(0)\leq\frac{1}{100m}$, there must exist $j\in\{1,2,\ldots,m-1\}$ such that $\left|\E_{x\sim\mu}[\omega_m^{jx}]\right|\geq\frac{1}{sf(s)}$, where $s=\frac{m}{\gcd(j,m)}$ is the order of $\omega_m^j$ for $\omega_m=e^{\frac{2\pi}{m}i}$, and $f:\mathbb{Z}^+\mapsto\mathbb{R}$ is any function satisfying $\sum_{s=2}^\infty\frac{1}{f(s)}\leq0.99$.
\end{lem}

\begin{proof}
By setting $x=0$ in~(\ref{eq:fourier}), we have
\[\mu(0)=\sum_{j=0}^{m-1}\hat{\mu}(j)\omega_m^{j\cdot0}=\sum_{j=0}^{m-1}\hat{\mu}(j)=\frac{1}{m}+\frac{1}{m}\sum_{j=1}^{m-1}\E_{x\sim\mu}[\omega_m^{-jx}].\]
Therefore
\begin{equation} \label{eq:biaslb}
\sum_{j=1}^{m-1}\left|\E_{x\sim\mu}[\omega_m^{jx}]\right| \geq \left|\sum_{j=1}^{m-1}\E_{x\sim\mu}[\omega_m^{jx}]\right| = \left|\sum_{j=1}^{m-1}\E_{x\sim\mu}[\omega_m^{-jx}]\right| = m\cdot\left|\mu(0)-\frac{1}{m}\right|\geq0.99.
\end{equation}

For every $d\mid m$ ($1\leq d\leq m-1$), define $T_d=\{j\mid\gcd(j,m)=d,1\leq j\leq m-1\}$. For all $j\in T_d$, the order of $\omega_m^j$ is $s_d=\frac{m}{d}$ ($2\leq s_d\leq m$). We also see $T_d=\{k\cdot d\mid1\leq k<s_d, \gcd(k,s_d)=1\}$, hence $|T_d|=\varphi(s_d)<s_d$.

If the lemma was not true, we have
\[\sum_{j=1}^{m-1}\left|\E_{x\sim\mu}[\omega_m^{jx}]\right|=\sum_{d\mid m \atop d<m}\left(\sum_{j\in T_d}\left|\E_{x\sim\mu}[\omega_m^{jx}]\right|\right)<\sum_{d\mid m \atop d<m}\left(s_d\cdot\frac{1}{s_df(s_d)}\right)<\sum_{s=2}^{\infty}\frac{1}{f(s)}\leq0.99.\]
This violates inequality~(\ref{eq:biaslb}). Thus the lemma is proved.
\end{proof}

\subsection{Notations and Facts about MV Families}

We use $\langle\cdot,\cdot\rangle$ to denote the inner product over $\mathbb{Z}$ between two vectors. In all calculations, we identify $\mathbb{Z}_m$ as $\{0,1,\ldots,m-1\}$ and treat the numbers as on $\mathbb{Z}$. Conventionally, we consider $a\bmod 1$ to be $0$ for any integer $a$.

\begin{notation}
Let $r$ be a positive integer. For an integer $v$, define $\vmod{v}{r}\in\{0,1,\ldots,r-1\}$ to be $v$ modulo $r$. For a vector $\vec{v}=(v_1,v_2,\ldots,v_n)$, define $\vmod{\vec{v}}{r}=(\vmod{v_1}{r},\vmod{v_2}{r},\ldots,\vmod{v_n}{r})$. For a list of vectors $V=(\vec{v}_1,\vec{v}_2,\ldots,\vec{v}_t)$, define $\vmod{V}{r}=(\vmod{\vec{v}_1}{r},\vmod{\vec{v}_2}{r},\ldots,\vmod{\vec{v}_t}{r})$.
\end{notation}

\begin{notation}
Let $r$ be a positive integer. For an integer $v$, define $\vdiv{v}{r}\in\mathbb{Z}$ to be $(v-\vmod{v}{r})/r$. For a vector $\vec{v}=(v_1,v_2,\ldots,v_n)$, define $\vdiv{\vec{v}}{r}=(\vdiv{v_1}{r},\vdiv{v_2}{r},\ldots,\vdiv{v_n}{r})$. Thus $\vec{v}=r\vdiv{\vec{v}}{r}+\vmod{\vec{v}}{r}$ for any vector $\vec{v}$. For a list of vectors $V=(\vec{v}_1,\vec{v}_2,\ldots,\vec{v}_t)$, define $\vdiv{V}{r}=(\vdiv{\vec{v}_1}{r},\vdiv{\vec{v}_2}{r},\ldots,\vdiv{\vec{v}_t}{r})$.
\end{notation}

\begin{define}
Let $U=(\vec{u}_1,\vec{u}_2,\ldots,\vec{u}_t)$ and $V=(\vec{v}_1,\vec{v}_2,\ldots,\vec{v}_t)$ be two lists of vectors in $\mathbb{Z}_m^n$. $(U,V)$ is a {\em matching vector family} if $\langle\vec{u}_i,\vec{v}_i\rangle\equiv0\pmod m$ for all $i\in[t]$ and $\langle\vec{u}_i,\vec{v}_j\rangle\not\equiv0\pmod m$ for all $i\neq j\in[t]$. The number $t$ is the size of the MV family and is denoted by $|(U,V)|$.
\end{define}

\begin{claim} \label{claim:twinfree}
For an MV family $(U,V)$ where $U=(\vec{u}_1,\vec{u}_2,\ldots,\vec{u}_t)$, $V=(\vec{v}_1,\vec{v}_2,\ldots,\vec{v}_t)$ and $i\neq j\in[t]$, we have $\vec{u}_i\neq\vec{u}_j$ and $\vec{v}_i\neq\vec{v}_j$.
\end{claim}

\begin{proof}
Assume $\vec{u}_i=\vec{u}_j$ for $i\neq j$, we have $\langle\vec{u}_i,\vec{v}_j\rangle=\langle\vec{u}_i,\vec{v}_i\rangle\equiv0\pmod m$. This violates the definition of MV family.
\end{proof}

\begin{notation}
Let $U,V,U',V'$ be 4 lists of vectors in $\mathbb{Z}_m^n$, and say $U=(\vec{u}_1,\vec{u}_2,\ldots,\vec{u}_t)$, $V=(\vec{v}_1,\vec{v}_2,\ldots,\vec{v}_t)$. We write $(U',V')\subseteq(U,V)$ if there exists a set $T\subseteq[t]$ such that $U'=(\vec{u}_i:i\in T)$ and $V'=(\vec{v}_i:i\in T)$. Observe that if $(U,V)$ is an MV family, so is $(U',V')$.
\end{notation}

\begin{define}
$(r_1,r_2,r_3)$ is a {\em partition} of $m$ if $r_1,r_2,r_3\in\mathbb{Z}^+$ and $r_1r_2r_3=m$. ($r_1,r_2,r_3$ are not assumed to be coprime.)
\end{define}

\begin{define}
For an MV family $(U,V)$ where $U=(\vec{u}_1,\vec{u}_2,\ldots,\vec{u}_t)$ and $V=(\vec{v}_1,\vec{v}_2,\ldots,\vec{v}_t)$, we say $(U,V)$ {\em respects} $(r_1,r_2,r_3)$, where $(r_1,r_2,r_3)$ is a partition of $m$, if the following conditions are satisfied:
\begin{enumerate}
\item $\exists\vec{u}_0\in\mathbb{Z}_{r_1}^n$ such that $\vmod{\vec{u}_i}{r_1}=\vec{u}_0$ for all $i\in[t]$,
\item $\exists\vec{v}_0\in\mathbb{Z}_{r_2}^n$ such that $\vmod{\vec{v}_i}{r_2}=\vec{v}_0$ for all $i\in[t]$,
\item $\langle\vdiv{\vec{u}_i}{r_1},\vec{v}_0\rangle$ modulo $r_2$ is the same for all $i\in[t]$,
\item $\langle\vec{u}_0,\vdiv{\vec{v}_i}{r_2}\rangle$ modulo $r_1$ is the same for all $i\in[t]$.
\end{enumerate}
\end{define}

\begin{claim} \label{claim:0diag}
If an MV family $(U,V)$ respects $(r_1,r_2,r_3)$, then $\langle\vec{u}_i,\vec{v}_j\rangle\equiv0\pmod{r_1r_2}$ for all $\vec{u}_i\in U, \vec{v}_j\in V$.
\end{claim}

\begin{proof}
Let $\vec{u}_0=\vmod{\vec{u}_i}{r_1}$ and $\vec{v}_0=\vmod{\vec{v}_j}{r_2}$. They are fixed for all $\vec{u}_i\in U$ and $\vec{v}_j\in V$. We have
\begin{eqnarray*}
\langle\vec{u}_i,\vec{v}_j\rangle & = & \langle r_1\vdiv{\vec{u}_i}{r_1}+\vec{u}_0,r_2\vdiv{\vec{v}_j}{r_2}+\vec{v}_0\rangle \\
& = & r_1r_2\langle\vdiv{\vec{u}_i}{r_1},\vdiv{\vec{v}_j}{r_2}\rangle+r_1\langle\vdiv{\vec{u}_i}{r_1},\vec{v}_0\rangle+r_2\langle\vec{u}_0,\vdiv{\vec{v}_j}{r_2}\rangle+\langle\vec{u}_0,\vec{v}_0\rangle.
\end{eqnarray*}
The first term is $0$ modulo $r_1r_2$. The second term is fixed modulo $r_1r_2$ because $\langle\vdiv{\vec{u}_i}{r_1},\vec{v}_0\rangle$ is fixed modulo $r_2$. Similarly, the third term is also a constant modulo $r_1r_2$. Therefore $\langle\vec{u}_i,\vec{v}_j\rangle$ modulo $r_1r_2$ is the same for all $\vec{u}_i\in U$ and $\vec{v}_j\in V$. Note that when $i=j$, $\langle\vec{u}_i,\vec{v}_j\rangle\equiv0\pmod{r_1r_2}$ since $(U,V)$ is an MV family. Therefore $\langle\vec{u}_i,\vec{v}_j\rangle\equiv0\pmod{r_1r_2}$ for all $\vec{u}_i\in U, \vec{v}_j\in V$.
\end{proof}

\begin{claim} \label{claim:base}
Every MV family $(U,V)$ respects $(1,1,m)$.
\end{claim}

\begin{proof}
Let $\vec{u}_0$ and $\vec{v}_0$ be the zero vector. All the conditions are satisfied.
\end{proof}

\begin{claim} \label{claim:goal}
If an MV family $(U,V)$ respects $(r_1,r_2,1)$, then it must has size $1$.
\end{claim}

\begin{proof}
Since $r_1r_2=m$, by Claim~\ref{claim:0diag} we have $\langle\vec{u}_i,\vec{v}_j\rangle\equiv0\pmod{m}$ for all $\vec{u}_i\in U, \vec{v}_j\in V$. By the definition of MV family, the size of $(U,V)$ must be $1$.
\end{proof}

\section{Proof of the Main Lemma}\label{sec-prooflem}

Consider an MV family $(U,V)$, where $U=(\vec{u}_1,\vec{u}_2,\ldots,\vec{u}_t)$ and $V=(\vec{v}_1,\vec{v}_2,\ldots,\vec{v}_t)$. We pick $\vec{u}\in U$ and $\vec{v}\in V$ uniformly at random and consider the distribution of $\vmod{\langle\vec{u},\vec{v}\rangle}{m}$. The inner product is 0 with probability $1/t$. Thus the distribution is far from uniform when $t>>m$. We will take advantage of this fact and prove our key lemma. For an MV family $(U,V)$ respecting $(r_1,r_2,r_3)$, we can find a large subfamily and reduce $r_3$ to some smaller number.

Let $f:\mathbb{Z}^+\mapsto\mathbb{R}$ be a function satisfying $\sum_{s=2}^{\infty}\frac{1}{f(s)}\leq0.99$. We will specify $f(s)$ in later proofs.

\begin{lem} \label{lem:main}
If an MV family $(U,V)$ respects $(r_1,r_2,r_3)$ for some $r_3\geq2$ and $|(U,V)|=t\geq100m$, then there exists $s\mid r_3$ with $s\geq2$ and an MV family $(U',V')\subseteq(U,V)$ with $|(U',V')|\geq t/(s^{n/2+4}f(s)^2)$ that respects either $(r_1s,r_2,r_3/s)$ or $(r_1,r_2s,r_3/s)$.
\end{lem}

\begin{proof}
We prove the lemma in 4 steps.

\subsubsection*{Step 1: Finding a nice character with a large bias.}
By Claim~\ref{claim:0diag}, $\frac{\langle\vec{u},\vec{v}\rangle}{r_1r_2}$ is an integer for all $\vec{u}\in U,\vec{v}\in V$. We can also see $\frac{\langle\vec{u},\vec{v}\rangle}{r_1r_2}\equiv0\pmod{r_3}$ iff $\langle\vec{u},\vec{v}\rangle\equiv0\pmod m$. Consider the distribution of $\vmod{\left(\frac{\langle\vec{u},\vec{v}\rangle}{r_1r_2}\right)}{r_3}\in\mathbb{Z}_{r_3}$, where $\vec{u}$ and $\vec{v}$ are uniformly drawn from $U$ and $V$ respectively. We have
\[\Pr\left[\vmod{\left(\frac{\langle\vec{u},\vec{v}\rangle}{r_1r_2}\right)}{r_3}=0\right]=\Pr\Big[\langle\vec{u},\vec{v}\rangle\equiv0\pmod m\Big]=\frac{1}{t}\leq\frac{1}{100m}\leq\frac{1}{100r_3}.\]
Applying Lemma~\ref{lem:bias} on $\mathbb{Z}_{r_3}$, there exists a $j\in\{1,2,\ldots,r_3-1\}$ such that
\begin{equation} \label{eq:start}
\left|\E_{\vec{u}\sim U \atop \vec{v}\sim V}\left[\omega_{r_3}^{j\frac{\langle\vec{u},\vec{v}\rangle}{r_1r_2}}\right]\right|\geq\frac{1}{sf(s)},
\end{equation}
where $\omega_{r_3}=e^{\frac{2\pi i}{r_3}}$ and $s=\frac{r_3}{\gcd(j,r_3)}$ is the order of $\omega_{r_3}^j$. Note that we have dropped the modulo $r_3$ operation because $(\omega_{r_3}^j)^{r_3}=1$. It follows that
\[\E_{\vec{u},\widetilde{\vec{u}}\sim U \atop \vec{v}\sim V}\left[\omega_{r_3}^{j\frac{\langle\vec{u}-\widetilde{\vec{u}},\vec{v}\rangle}{r_1r_2}}\right]=\E_{\vec{v}\sim V}\left|\E_{\vec{u}\sim U}\left[\omega_{r_3}^{j\frac{\langle\vec{u},\vec{v}\rangle}{r_1r_2}}\right]\right|^2\geq\left|\E_{\vec{u}\sim U \atop \vec{v}\sim V}\left[\omega_{r_3}^{j\frac{\langle\vec{u},\vec{v}\rangle}{r_1r_2}}\right]\right|^2\geq\frac{1}{s^2f(s)^2}.\]
Therefore there exists a fixed $\widetilde{\vec{u}}\in U$ such that
\[\left|\E_{\vec{u}\sim U \atop \vec{v}\sim V}\left[\omega_{r_3}^{j\frac{\langle\vec{u}-\widetilde{\vec{u}},\vec{v}\rangle}{r_1r_2}}\right]\right|=\left|\E_{\vec{u}\sim U \atop \vec{v}\sim V}\left[\omega_{r_3}^{j\langle\frac{\vec{u}-\widetilde{\vec{u}}}{r_1},\vec{v}\rangle/r_2}\right]\right|\geq\frac{1}{s^2f(s)^2}.\]
Since $\vmod{\vec{u}}{r_1}=\vmod{\widetilde{\vec{u}}}{r_1}$, we have $\vec{u}-\widetilde{\vec{u}}=r_1(\vdiv{\vec{u}}{r_1}-\vdiv{\widetilde{\vec{u}}}{r_1})$. The above inequality can be written as 
\begin{equation} \label{eq:step1}
\left|\E_{\vec{u}\sim U \atop \vec{v}\sim V}\left[\omega_{r_3}^{j\langle\vdiv{\vec{u}}{r_1}-\vdiv{\widetilde{\vec{u}}}{r_1},\vec{v}\rangle/r_2}\right]\right|\geq\frac{1}{s^2f(s)^2}.
\end{equation}

\subsubsection*{Step 2: Partitioning  into buckets.}

We partition the set $U$ into buckets according to $\vdiv{\vec{u}}{r_1}-\vdiv{\widetilde{\vec{u}}}{r_1}$ modulo $s$: $U=\bigcup\limits_{\vec{w}\in\mathbb{Z}_s^n}B(\vec{w},U)$, where
\[\widetilde{B}(\vec{w},U)=\left\{\vec{u}\in U\Bigm|\vmod{\left(\vdiv{\vec{u}}{r_1}-\vdiv{\widetilde{\vec{u}}}{r_1}\right)}{s}=\vec{w}\right\}.\]

We also partition $V$ into buckets $B(\vec{w},V)=\{\vec{v}\in V\mid\vmod{(\vdiv{\vec{v}}{r_2})}{s}=\vec{w}\}$ for all $\vec{w}\in\mathbb{Z}_s^n$. Define $p_{\vec{w}}=|\widetilde{B}(\vec{w},U)|/t$ to be the density of $\widetilde{B}(\vec{w},U)$ and $q_{\vec{w}}=|B(\vec{w},V)|/t$ be the density of $B(\vec{w},V)$.

Picking $\vec{u}$ uniformly from $U$ can be equivalently considered as two steps: 1.~For each bucket $\widetilde{B}(\vec{w},U)$, pick a representative $\vec{u}_{\vec{w}}\in\widetilde{B}(\vec{w},U)$ uniformly; 2.~Pick one bucket according to the probability distribution $p_{\vec{w}}$, and output the representative. For inequality~(\ref{eq:step1}), we split the procedure of picking $\vec{u}\sim U$ into these two steps.
\begin{eqnarray*}
\frac{1}{s^2f(s)^2} & \leq & \left|\E_{\vec{u}\sim U \atop \vec{v}\sim V}\left[\omega_{r_3}^{j\langle\vdiv{\vec{u}}{r_1}-\vdiv{\widetilde{\vec{u}}}{r_1},\vec{v}\rangle/r_2}\right]\right| \\
& = & \left|\E_{\text{for each }\vec{w}, \atop \vec{u}_{\vec{w}}\sim\widetilde{B}(\vec{w},U)}\E_{\vec{w}\sim p_{\vec{w}}}\E_{\vec{v}\sim V}\left[\omega_{r_3}^{j\langle\vdiv{\vec{u}_{\vec{w}}}{r_1}-\vdiv{\widetilde{\vec{u}}}{r_1},\vec{v}\rangle/r_2}\right]\right| \\
& \leq & \E_{\text{for each }\vec{w}, \atop \vec{u}_{\vec{w}}\sim\widetilde{B}(\vec{w},U)}\left|\E_{\vec{w}\sim p_{\vec{w}}}\E_{\vec{v}\sim V}\left[\omega_{r_3}^{j\langle\vdiv{\vec{u}_{\vec{w}}}{r_1}-\vdiv{\widetilde{\vec{u}}}{r_1},\vec{v}\rangle/r_2}\right]\right|.
\end{eqnarray*}

There exists a fixed list of representatives from each bucket $(\vec{u}_{\vec{w}}\in B(\vec{w},U):\vec{w}\in\mathbb{Z}_s^n)$ such that
\begin{equation} \label{eq:rep}
\frac{1}{s^2f(s)^2} \leq \left|\E_{\vec{w}\sim p_{\vec{w}} \atop \vec{v}\sim V}\left[\omega_{r_3}^{j\langle\vdiv{\vec{u}_{\vec{w}}}{r_1}-\vdiv{\widetilde{\vec{u}}}{r_1},\vec{v}\rangle/r_2}\right]\right|.
\end{equation}

For every $\vec{w}\in\mathbb{Z}_s^n$ and $\vec{u}\in B(\vec{w},U)$, we use $\vec{u}'$ to denote the vector $\vdiv{(\vdiv{\vec{u}}{r_1}-\vdiv{\widetilde{\vec{u}}}{r_1})}{s}$. Thus
\[\vdiv{\vec{u}_{\vec{w}}}{r_1}-\vdiv{\widetilde{\vec{u}}}{r_1}=s\vec{u}_{\vec{w}}'+\vmod{(\vdiv{\vec{u}_{\vec{w}}}{r_1}-\vdiv{\widetilde{\vec{u}}}{r_1})}{s}=s\vec{u}_{\vec{w}}'+\vec{w}.\]

Hence inequality~(\ref{eq:rep}) can be written as
\begin{equation} \label{eq:step2}
\frac{1}{s^2f(s)^2} \leq \left|\E_{\vec{w}\sim p_{\vec{w}} \atop \vec{v}\sim V}\left[\omega_{r_3}^{j\langle s\vec{u}_{\vec{w}}'+\vec{w},\vec{v}\rangle/r_2}\right]\right|.
\end{equation}

\subsubsection*{Step 3: Finding a large bucket.}

By inequality~(\ref{eq:step2}),
\begin{eqnarray}
\left(\frac{1}{s^2f(s)^2}\right)^2 & \leq & \left|\sum_{\vec{w}\in\mathbb{Z}_s^n}\sum_{\vec{v}\in V}p_{\vec{w}}\cdot\frac{1}{t}\cdot\omega_{r_3}^{j\langle s\vec{u}_{\vec{w}}'+\vec{w},\vec{v}\rangle/r_2}\right|^2 \nonumber \\
& \leq & \left(\sum_{\vec{w}\in\mathbb{Z}_s^n}p_{\vec{w}}^2\right)\cdot\left(\sum_{\vec{w}\in\mathbb{Z}_s^n}\left|\sum_{\vec{v}\in V}\frac{1}{t}\cdot\omega_{r_3}^{j\langle s\vec{u}_{\vec{w}}'+\vec{w},\vec{v}\rangle/r_2}\right|^2\right) \nonumber \\
& = & \left(\sum_{\vec{w}\in\mathbb{Z}_s^n}p_{\vec{w}}^2\right)\cdot\left(\sum_{\vec{w}\in\mathbb{Z}_s^n}\sum_{\vec{v},\widetilde{\vec{v}}\in V}\frac{1}{t^2}\cdot\omega_{r_3}^{j\langle s\vec{u}_{\vec{w}}'+\vec{w},\vec{v}-\widetilde{\vec{v}}\rangle/r_2}\right) \nonumber \\
& = & \left(\sum_{\vec{w}\in\mathbb{Z}_s^n}p_{\vec{w}}^2\right)\cdot\left(\sum_{\vec{v},\widetilde{\vec{v}}\in V}\frac{1}{t^2}\cdot\sum_{\vec{w}\in\mathbb{Z}_s^n}\omega_{r_3}^{j\langle s\vec{u}_{\vec{w}}'+\vec{w},\vdiv{\vec{v}}{r_2}-\vdiv{\widetilde{\vec{v}}}{r_2}\rangle}\right) \nonumber \\
& = & \left(\sum_{\vec{w}\in\mathbb{Z}_s^n}p_{\vec{w}}^2\right)\cdot\left(\sum_{\vec{v},\widetilde{\vec{v}}\in V}\frac{1}{t^2}\cdot\sum_{\vec{w}\in\mathbb{Z}_s^n}\omega_{r_3}^{j\langle\vec{w},\vdiv{\vec{v}}{r_2}-\vdiv{\widetilde{\vec{v}}}{r_2}\rangle}\right) \nonumber \\
& = & \left(\sum_{\vec{w}\in\mathbb{Z}_s^n}p_{\vec{w}}^2\right)\cdot\left(\sum_{\vec{v},\widetilde{\vec{v}}\in V}\frac{1}{t^2}\cdot s^n\cdot{\scalebox{1.3}{1}}_{\vdiv{\vec{v}}{r_2}\neq\vdiv{\widetilde{\vec{v}}}{r_2}}\right) \nonumber \\
& = & \left(\sum_{\vec{w}\in\mathbb{Z}_s^n}p_{\vec{w}}^2\right)\cdot\left(\sum_{\vec{w}\in\mathbb{Z}_s^n}q_{\vec{w}}^2\right)\cdot s^n. \label{eq:cp}
\end{eqnarray}
In the last step we used the fact $\vdiv{\vec{v}}{r_2}\neq\vdiv{\widetilde{\vec{v}}}{r_2}$ for two $\vec{v},\widetilde{\vec{v}}\in V$. This can be seen by contrapositive. If $\vdiv{\vec{v}}{r_2}=\vdiv{\widetilde{\vec{v}}}{r_2}$, we have $\vec{v}=\widetilde{\vec{v}}$ since $\vmod{\vec{v}}{r_2}=\vmod{\widetilde{\vec{v}}}{r_2}$. This contradicts Claim~\ref{claim:twinfree}.

By~(\ref{eq:cp}), we see either $\sum p_{\vec{w}}^2\geq1/(s^{n/2+2}f(s)^2)$ or $\sum q_{\vec{w}}^2\geq1/(s^{n/2+2}f(s)^2)$. Without loss of generality, assume $\sum p_{\vec{w}}^2\geq1/(s^{n/2+2}f(s)^2)$. By
\[\max\{p_{\vec{w}}\}=\max\{p_{\vec{w}}\}\cdot\sum p_{\vec{w}}\geq\sum p_{\vec{w}}^2,\]
there exists a bucket $\widetilde{B}(\vec{w}_0,U)$ with size at least $t/(s^{n/2+2}f(s)^2)$. Let $\widetilde{U}$ be that bucket, and $\widetilde{V}$ be the subset of $V$ with the same indices. Then $(\widetilde{U},\widetilde{V})\subseteq(U,V)$ is an MV family of size at least $t/(s^{n/2+2}f(s)^2)$. Next, we will find a subfamily $(U',V')\subseteq(\widetilde{U},\widetilde{V})$ that respects $(r_1s,r_2,r_3/s)$.

\subsubsection*{Step 4: Analyzing the elements in the large bucket.}

Let $\vec{u}_0$ and $\vec{v}_0$ denote $\vmod{\vec{u}}{r_1}$ and $\vmod{\vec{v}}{r_2}$ respectively for $\vec{u}\in U,\vec{v}\in V$. For every $\vec{u}\in \widetilde{U}$, we know $\vmod{(\vdiv{\vec{u}}{r_1}-\vdiv{\widetilde{\vec{u}}}{r_1})}{s}$  equals the same vector $\vec{w}_0$ by the definition of the bucket. Therefore $\vdiv{\vec{u}}{r_1}-\vdiv{\widetilde{\vec{u}}}{r_1}=s\vec{u}'+\vec{w}_0$ and
\begin{equation} \label{eq:expressu}
\vec{u}=r_1\vdiv{\vec{u}}{r_1}+\vec{u}_0=r_1(\vdiv{\widetilde{\vec{u}}}{r_1}+s\vec{u}'+\vec{w}_0)+\vec{u}_0=r_1s\vec{u}'+\left(r_1\vdiv{\widetilde{\vec{u}}}{r_1}+r_1\vec{w}_0+\vec{u}_0\right).
\end{equation}
We can see $\vmod{\vec{u}}{r_1s}=r_1\vdiv{\widetilde{\vec{u}}}{r_1}+r_1\vec{w}_0+\vec{u}_0$ is the same for all $\vec{u}\in \widetilde{U}$. Also $\vmod{\vec{v}}{r_2}=\vec{v}_0$ is the same for all $\vec{v}\in \widetilde{V}$. These two conditions are still satisfied for any subfamily of $(\widetilde{U},\widetilde{V})$. It suffices to find $(U',V')\subseteq(\widetilde{U},\widetilde{V})$ such that
\begin{itemize}
\item $\langle\vdiv{\vec{u}}{r_1s},\vec{v}_0\rangle$ modulo $r_2$ is the same for all $\vec{u}\in U'$. By~(\ref{eq:expressu}) we have $\vdiv{\vec{u}}{r_1s}=\vec{u}'$, so we need $\langle\vec{u}',\vec{v}_0\rangle$ modulo $r_2$ to be the same for all $\vec{u}\in U'$.
\item $\langle r_1\vdiv{\widetilde{\vec{u}}}{r_1}+r_1\vec{w}+\vec{u}_0,\vdiv{\vec{v}}{r_2}\rangle$ modulo $r_1s$ is the same for all $\vec{v}\in V'$.
\end{itemize}

Since $\langle\vdiv{\vec{u}}{r_1},\vec{v}_0\rangle=\langle s\vec{u}'+\vdiv{\widetilde{\vec{u}}}{r_1}+\vec{w},\vec{v}_0\rangle$ modulo $r_2$ is the same for all $\vec{u}\in U$ by $(U,V)$ respecting $(r_1,r_2,r_3)$, we can see that $s\langle\vec{u}',\vec{v}_0\rangle$ modulo $r_2$ is the same for all $\vec{u}\in U$. Hence there are $\gcd(s,r_2)$ possible values for $\langle\vec{u}',\vec{v}_0\rangle$ modulo $r_2$. We pick the most frequent value $c_1$ and keep only the vectors with $\langle\vec{u}',\vec{v}_0\rangle\equiv c_1\pmod{r_2}$ in $\widetilde{U}$ and the corresponding vectors in $\widetilde{V}$.

Since $\langle\vec{u}_0,\vdiv{\vec{v}}{r_2}\rangle$ modulo $r_1$ is the same for all $\vec{v}\in V$ by $(U,V)$ respecting $(r_1,r_2,r_3)$, we can see that there are $s$ possible values for $\langle\vec{u}_0,\vdiv{\vec{v}}{r_2}\rangle$ modulo $sr_1$. We pick the most frequent value $c_2$ and keep only the vectors with $\langle\vec{u}_0,\vdiv{\vec{v}}{r_2}\rangle\equiv c_2\pmod{sr_1}$ in $\widetilde{U}$ and the corresponding vectors in $\widetilde{V}$.

After the above two steps, the MV family has size at least
\begin{equation} \label{eq:last}
\frac{|(\widetilde{U},\widetilde{V})|}{\gcd(s,r_2)\cdot s}\geq\frac{|(\widetilde{U},\widetilde{V})|}{s^2}\geq\frac{t}{s^{n/2+4}f(s)^2}.
\end{equation}
And this is the required $(U',V')$.
\end{proof}

\section{Proof of Theorems~\ref{thm:main} and \ref{thm-MVcodes}}\label{sec-proofthm}

We now prove Theorem~\ref{thm:main} by repeatedly applying Lemma~\ref{lem:main}.

\begin{proof}[Proof of Theorem~\ref{thm:main}]
By Claim~\ref{claim:base}, $(U,V)$ is good with respect to $(1,1,m)$. Initially we set $r_1=1$, $r_2=1$ and $r_3=m$. By Lemma~\ref{lem:main}, we there is a subfamily that respects $(r_1',r_2',r_3')$, where $r_1'r_2'r_3'=m$ and $r_3'<m$. We repeatedly apply Lemma~\ref{lem:main}. Each round $r_3$ is reduced by some factor. We can continue this procedure until either $r_3=1$ or the size of the MV family becomes less than $100m$. For the case $r_3=1$, the size of the MV family is also less than $100m$ by Claim~\ref{claim:goal}. Say there are $k$ rounds, and in each round we divide $r_3$ by $s_1,s_2,\ldots,s_k$ respectively. We have $s_1s_2\cdots s_k\leq m$ and in the $i$th round ($i\in[k]$), the size of the MV family is decreased by a factor at most $s_i^{n/2+4}f(s_i)^2$. Therefore the original size is upper bounded by
\[|(U,V)|\leq100m\cdot\prod_{i=1}^k s_i^{n/2+4}f(s_i)^2\leq 100m\cdot m^{n/2+4}\cdot\prod_{i=1}^kf(s_i)^2=100m^{n/2+5}\prod_{i=1}^kf(s_i)^2.\]

Pick $f(s)=s^{1.735}$, we can verify $\sum_{s=2}^{\infty}\frac{1}{f(s)}\leq0.99$. Therefore $|(U,V)|\leq100m^{n/2+5}(m^{1.735})^2=100m^{n/2+8.47}$.
\end{proof}

Combining with the lower bound $m^{n-1+o_m(1)}$ proved in \cite{DGY11}, we can give a universal lower bound for the length of the MV code in \cite{DGY11}. This is a restatement of Theorem~\ref{thm-MVcodes} stated in the introduction.

\begin{cor}\label{cor-MVcodes}
Any MV code (as constructed in \cite{DGY11}) has encoding length at least $N>K^{\frac{19}{18}}$, where $K$ is the message length regardless of the query complexity.
\end{cor}

\begin{proof}
Given an MV family in $\mathbb{Z}_m^n$ with size $t$, we can encode a message of length $K=t$ into a codeword of length $N=m^n$.

If $n\geq19$, by Theorem~\ref{thm:main} we have $K\leq m^{n/2+8.47}$. Hence $K\leq m^{(1/2+8.47/19)n}<m^{\frac{18}{19}n}=N^{\frac{18}{19}}$ and $N>K^{\frac{19}{18}}$.

If $n\leq18$, it was shown in \cite{DGY11} that $K\leq m^{n-1+o_m(1)}$. Hence $K<m^{n-\frac{18}{19}}\leq m^{n-\frac{n}{19}}=m^{\frac{18}{19}n}=N^{\frac{18}{19}}$ and $N>K^{\frac{19}{18}}$. Note that here we assumed $m$ is sufficient large. This is reasonable because we are considering encoding an arbitrarily long message and $K$ is sufficiently large.
\end{proof}

\section{The case of distinct prime factors}\label{sec-prime}

If $m$ is a product of distinct primes, the bound can be improved to $m^{n/2+4+o_m(1)}$. The proof follows the same outline as general composite $m$.

\begin{thm}
Let $m$ be a product of distinct primes. For every MV family $(U,V)$ in $\mathbb{Z}_m^n$,  $|(U,V)|\leq100m^{n/2+4+o_m(1)}$, where $o_m(1)$ goes to $0$ as $m$ grows.
\end{thm}

\begin{proof}
The proof is similar to Theorem~\ref{thm:main}. We only sketch the changes here.

First, we improve the size of $(U',V')$ found in Lemma~\ref{lem:main} to $t/(s^{n/2+2}f(s)^2)$. Since $m$ is a product of distinct primes, $r_1$ and $r_2$ must be coprime to $s$, where $s$ is the number in inequality~(\ref{eq:start}). Let $\tau_1$ and $\tau_2$ be integers that $\tau_1r_1\equiv1\pmod s$ and $\tau_2r_2\equiv1\pmod s$, we have
\[\omega_{r_3}^{j\frac{\langle\vec{u},\vec{v}\rangle}{r_1r_2}}=\omega_{r_3}^{j\langle\vec{u},\vec{v}\rangle\tau_1\tau_2}.\]
We partition $U$ and $V$ into buckets according to $\vec{u}$ modulo $s$ and $\vec{v}$ modulo $s$: $U=\bigcup\limits_{\vec{w}\in\mathbb{Z}_s^n}B(\vec{w},U)$ and $V=\bigcup\limits_{\vec{w}\in\mathbb{Z}_s^n}B(\vec{w},V)$, where
\[B(\vec{w},U)=\{\vec{u}\in U\mid\vmod{\vec{u}}{s}=\vec{w}\}\]
and
\[B(\vec{w},V)=\{\vec{v}\in V\mid\vmod{\vec{v}}{s}=\vec{w}\}.\]
We still use $p_{\vec{w}}$ to denote $|B(\vec{w},U)|/t$ and $q_{\vec{w}}$ to denote $|B(\vec{w},V)|/t$. By inequality~(\ref{eq:start}),
\begin{eqnarray*}
\left(\frac{1}{sf(s)}\right)^2 & \leq & \left|\E_{\vec{u}\sim U \atop \vec{v}\sim V}\left[\omega_{r_3}^{j\langle\vec{u},\vec{v}\rangle\tau_1\tau_2}\right]\right|^2 \\
& = & \left|\sum_{\vec{w}\in\mathbb{Z}_s^n}\sum_{\vec{v}\in V}p_{\vec{w}}\cdot\frac{1}{t}\cdot\omega_{r_3}^{j\langle\vec{w},\vec{v}\rangle\tau_1\tau_2}\right|^2 \\
& \leq & \left(\sum_{\vec{w}\in\mathbb{Z}_s^n}p_{\vec{w}}^2\right)\cdot\left(\sum_{\vec{w}\in\mathbb{Z}_s^n}\left|\sum_{\vec{v}\in V}\frac{1}{t}\cdot\omega_{r_3}^{j\langle\vec{w},\vec{v}\rangle\tau_1\tau_2}\right|^2\right) \\
& = &  \left(\sum_{\vec{w}\in\mathbb{Z}_s^n}p_{\vec{w}}^2\right)\cdot\left(\sum_{\vec{v},\widetilde{\vec{v}}\in V}\frac{1}{t^2}\cdot\sum_{\vec{w}\in\mathbb{Z}_s^n}\omega_{r_3}^{j\langle\vec{w},\vec{v}-\widetilde{\vec{v}}\rangle\tau_1\tau_2}\right) \\
& = & \left(\sum_{\vec{w}\in\mathbb{Z}_s^n}p_{\vec{w}}^2\right)\cdot\left(\sum_{\vec{w}\in\mathbb{Z}_s^n}q_{\vec{w}}^2\right)\cdot s^n.
\end{eqnarray*}
We can see either $\sum p_{\vec{w}}^2\geq1/(s^{n/2+1}f(s))$ or $\sum q_{\vec{w}}^2\geq1/(s^{n/2+1}f(s))$. Without loss of generality, assume $\sum p_{\vec{w}}^2\geq1/(s^{n/2+1}f(s))$. By
\[\max\{p_{\vec{w}}\}=\max\{p_{\vec{w}}\}\cdot\sum p_{\vec{w}}\geq\sum p_{\vec{w}}^2,\]
there exists a bucket with size $|B(\vec{w},U)|\geq t/(s^{n/2+1}f(s))$. Let $\widetilde{U}$ be that bucket, and $\widetilde{V}$ be the subset of $V$ with the same indices. Then $(\widetilde{U},\widetilde{V})\subseteq(U,V)$ is an MV family of size at least $t/(s^{n/2+1}f(s))$.

Next we find $(U',V')\subseteq(\widetilde{U},\widetilde{V})$ using the same method  as in Lemma~\ref{lem:main}. By inequality~(\ref{eq:last}),
\[|(U',V')|\geq\frac{|(\widetilde{U},\widetilde{V})|}{\gcd(s,r_2)\cdot s}=\frac{|(\widetilde{U},\widetilde{V})|}{s}\geq\frac{t}{s^{n/2+2}f(s)}.\]

At last, we use the proof of Theorem~\ref{thm:main} except $f(s)=\frac{1}{3s\ln^2s}$. One can verify $\sum_{s=2}^{\infty}\frac{1}{f(s)}<0.99$. Let $s_1,s_2,\ldots,s_k$ be the numbers divided from $r_3$ in each round, by the proof of Theorem~\ref{thm:main},
\[|(U,V)|\leq100m\prod_{i=1}^ks_i^{n/2+2}f(s_i)\leq100m^{n/2+3}\prod_{i=1}^k(3s_i\ln^2s_i)\leq100m^{n/2+4}\prod_{i=1}^k(3\ln^2s_i).\]
For a sufficiently large integer $s$, we have $3\ln^2s<s^{\epsilon}$, where $\epsilon$ is an arbitrary fixed small number. When $m\to\infty$, all $s_1,s_2,\ldots,s_k$ except a constant number of them must be that large. Take $\epsilon\rightarrow0$, we have $|(U,V)|\leq m^{n/2+4+o_m(1)}$.
\end{proof}

\bibliographystyle{alpha}

\bibliography{mvbound}

\newcommand{\etalchar}[1]{$^{#1}$}
\begin{thebibliography}{GKST02}

\bibitem[BBR94]{BBR}
David A.~Mix Barrington, Richard Beigel, and Steven Rudich.
\newblock Representing boolean functions as polynomials modulo composite
  numbers.
\newblock In {\em Computational Complexity}, pages 455--461, 1994.

\bibitem[BDL13]{BDL13}
Abhishek Bhowmick, Zeev Dvir, and Shachar Lovett.
\newblock New bounds on matching vector families.
\newblock In {\em 45th ACM Symposium on Theory of Computing (STOC)}, 2013.

\bibitem[BET10]{BET10}
Avraham Ben{-}Aroya, Klim Efremenko, and Amnon Ta{-}Shma.
\newblock Local list decoding with a constant number of queries.
\newblock In {\em 51st IEEE Symposium on Foundations of Computer Science
  (FOCS)}, pages 715--722, 2010.

\bibitem[BF98]{BF98}
L\'aszl\'o Babai and Peter Frankl.
\newblock {\em Linear algebra methods in combinatorics}.
\newblock 1998.

\bibitem[CFL{\etalchar{+}}10]{CFL+10}
Yeow~Meng Chee, Tao Feng, San Ling, Huaxiong Wang, and Liang~Feng Zhang.
\newblock Query-efficient locally decodable codes of subexponential length.
\newblock {\em Electronic Colloquium on Computational Complexity (ECCC)},
  TR10-173, 2010.

\bibitem[DGY11]{DGY11}
Zeev Dvir, Parikshit Gopalan, and Sergey Yekhanin.
\newblock Matching vector codes.
\newblock {\em SIAM J. Comput.}, 40(4):1154--1178, 2011.

\bibitem[Efr09]{Efr09}
Klim Efremenko.
\newblock 3-query locally decodable codes of subexponential length.
\newblock In {\em 41st ACM Symposium on Theory of Computing (STOC)}, pages
  39--44, 2009.

\bibitem[GKST02]{GKST02}
Oded Goldreich, Howard Karloff, Leonard~J. Schulman, and Luca Trevisan.
\newblock Lower bounds for linear locally decodable codes and private
  information retrieval.
\newblock In {\em 17th IEEE Computational Complexity Conference (CCC)}, pages
  175--183, 2002.

\bibitem[Gro00]{Gro00}
Vince Grolmusz.
\newblock Superpolynomial size set-systems with restricted intersections mod 6
  and explicit ramsey graphs.
\newblock {\em Combinatorica}, 20(1):71--86, 2000.

\bibitem[IS10]{IS10}
Toshiya Itoh and Yasuhiro Suzuki.
\newblock Improved constructions for query-efficient locally decodable codes of
  subexponential length.
\newblock {\em IEICE Transactions on Information and Systems},
  E93-D(2):263--270, 2010.

\bibitem[KdW04]{KdW04}
Iordanis Kerenidis and Ronald de~Wolf.
\newblock Exponential lower bound for 2-query locally decodable codes via a
  quantum argument.
\newblock {\em Journal of Computer and System Sciences}, 69(3):395--420, 2004.

\bibitem[KT00]{KT00}
Jonathan Katz and Luca Trevisan.
\newblock On the efficiency of local decoding procedures for error-correcting
  codes.
\newblock In {\em 32nd ACM Symposium on Theory of Computing (STOC)}, pages
  80--86, 2000.

\bibitem[KY09]{KY09}
Kiran~S. Kedlaya and Sergey Yekhanin.
\newblock Locally decodable codes from nice subsets of finite fields and prime
  factors of {M}ersenne numbers.
\newblock {\em SIAM J. Comput.}, 38(5):1952--1969, 2009.

\bibitem[Rag07]{Rag07}
Prasad Raghavendra.
\newblock A note on {Y}ekhanin's locally decodable codes.
\newblock {\em Electronic Colloquium on Computational Complexity (ECCC)},
  TR07-016, 2007.

\bibitem[TV07]{TV}
Terrence Tao and Van~H. Vu.
\newblock {\em Additive Combinatorics}.
\newblock 2007.

\bibitem[Woo07]{Woo07}
David~P. Woodruff.
\newblock New lower bounds for general locally decodable codes.
\newblock {\em Electronic Colloquium on Computational Complexity (ECCC)},
  TR07-006, 2007.

\bibitem[Woo10]{Woo10}
David~P. Woodruff.
\newblock A quadratic lower bound for three-query linear locally decodable
  codes over any field.
\newblock In {\em Proceedings of the 13th international conference on
  Approximation, and 14 the International conference on Randomization, and
  combinatorial optimization: algorithms and techniques}, APPROX/RANDOM'10,
  pages 766--779, Berlin, Heidelberg, 2010. Springer-Verlag.

\bibitem[Yek08]{Yek08}
Sergey Yekhanin.
\newblock Towards 3-query locally decodable codes of subexponential length.
\newblock {\em Journal of the ACM}, 55(1):1--16, 2008.

\bibitem[YGK12]{YGK12}
Chen Yuan, Qian Guo, and Haibin Kan.
\newblock A novel elementary construction of matching vectors.
\newblock {\em Information Processing Letters}, 112(12):494--496, 2012.

\end{thebibliography}

\end{spacing}
\end{document}